\documentclass[11pt]{amsart}
\usepackage{amscd,amssymb}
\usepackage{amsthm,amsmath,enumerate,amssymb,wasysym}
\usepackage[matrix,arrow]{xy}
\usepackage{enumerate}
\usepackage{longtable}
\usepackage{comment}
\usepackage{fancyhdr}
\usepackage{supertabular}
\usepackage{setspace}
\usepackage{color}
\usepackage{longtable}

\theoremstyle{definition}

\newtheorem*{example*}{Example}

\newtheorem{theorem}[equation]{Theorem}
\newtheorem{lemma}[equation]{Lemma}
\newtheorem{corollary}[equation]{Corollary}
\newtheorem{proposition}[equation]{Proposition}
\newtheorem{conjecture}[equation]{Conjecture}
\newtheorem*{conjecture*}{Conjecture}

\newtheorem*{question*}{Question}

\newtheorem*{problem*}{Problem}
\newtheorem*{theorem*}{Theorem}

\theoremstyle{remark}
\newtheorem{remark}[equation]{Remark}
\newtheorem*{remark*}{Remark}

\makeatletter\@addtoreset{equation}{section} \makeatother

\author{Ivan Cheltsov, Jihun Park, Constantin Shramov}

\title{Delta invariants of singular del Pezzo surfaces}

\begin{document}

\begin{abstract}
We use  the methods introduced by Cheltsov--Rubinstein--Zhang in \cite{CRZ18} to
estimate $\delta$-invariants of the seven singular del Pezzo surfaces with quotient singularities
studied by Cheltsov--Park--Shramov in \cite{CheltsovParkShramovJGA} that have $\alpha$-invariants less than  $\frac{2}{3}$.
As a result, we verify that each of these surfaces admits an orbifold K\"ahler--Einstein metric.
\end{abstract}


\address{ \emph{Ivan Cheltsov}\newline \textnormal{School of Mathematics, The
University of Edinburgh
\newline \medskip  Edinburgh EH9 3JZ, UK
\newline
National Research University Higher School of Economics, Laboratory of Algebraic Geometry,
\newline
6 Usacheva street, Moscow, 117312, Russia
\newline
\texttt{I.Cheltsov@ed.ac.uk}}}

\address{ \emph{Jihun Park}\newline \textnormal{Center for Geometry and Physics, Institute for Basic Science
\newline \medskip 77 Cheongam-ro, Nam-gu, Pohang, Gyeongbuk, 37673, Korea \newline
Department of
Mathematics, POSTECH
\newline
77 Cheongam-ro, Nam-gu, Pohang, Gyeongbuk,  37673, Korea \newline
\texttt{wlog@postech.ac.kr}}}

\address{\emph{Constantin Shramov}\newline \textnormal{Steklov Mathematical Institute of Russian Academy of Sciences
\newline
\medskip 8 Gubkina street, Moscow, 119991, Russia
\newline
National Research University Higher School of Economics, Laboratory of Algebraic Geometry,
\newline
6 Usacheva street, Moscow, 117312, Russia
\newline \texttt{costya.shramov@gmail.com }}}

\maketitle

All varieties are assumed to be complex, projective and normal unless otherwise stated.

\section{Introduction}
\label{section:intro}

Let $S_d$ be a quasismooth and well-formed hypersurface in $\mathbb{P}(a_{0},a_{1},a_{2},a_{3})$ of degree $d$,
where~\mbox{$a_0\leqslant a_1\leqslant a_2\leqslant a_4$}.
Then $S_d$ is given by a quasihomogeneous polynomial
equation of degree $d$
$$
f\big(x,y,z,t\big)=0\subset\mathbb{P}(a_{0},a_{1},a_{2},a_{3})\cong\mathrm{Proj}\Big(\mathbb{C}\big[x,y,z,t\big]\Big),
$$
where $\mathrm{wt}(x)=a_{0}$, $\mathrm{wt}(y)=a_{1}$, $\mathrm{wt}(z)=a_{2}$, $\mathrm{wt}(t)=a_{3}$.
Here, being quasismooth simply means that the above equation defines a hypersurface that is singular only at the origin  in $\mathbb{C}^4$,
which implies that $S_d$ has at most cyclic quotient singularities.
On the other hand, being well-formed implies that
$$
K_{S_d}\sim_{\mathbb{Q}}\mathcal{O}_{\mathbb{P}(a_{0},a_{1},a_{2},a_{3})}\big(d-a_0-a_1-a_2-a_3\big),
$$
see~\cite[Theorem 3.3.4]{Do82}, \cite[6.14]{IF00}.

Put $I=a_0+a_1+a_2+a_3-d$ and suppose that $I$ is positive.
Then $S_d$ is a del Pezzo surfaces with at most quotient singularities.
If $S_d$ is smooth, then it always admits a K\"ahler--Einstein metric by \cite{TianSurface} (see also \cite{Cheltsov,Shi,Tosatti,OdakaSpottiSun}).
Singular del Pezzo surfaces with orbifold K\"ahler--Einstein metrics drew attention from Riemannian geometers
because they may lift to Sasakian--Einstein 5-manifolds through $S^1$-bundle structures.
Through this passage, Boyer, Galicki and Nakamaye yielded a significant amount of examples towards classification of simply-connected Sasakian--Einstein $5$-manifolds (see \cite{BoGaNa03,BoyerBook}).

For $I=1$, Johnson and Koll\'ar
presented an algorithm in \cite{JoKo01b} that produces the (infinite) list of all possibilities for the quintuple $(a_{0},a_{1},a_{2},a_{3},d)$
They also proved the following result:

\begin{theorem}[{\cite[Theorem~8]{JoKo01b}}]
\label{theorem:Johnson-Kollar}
Suppose that $S_d$ with $I=1$ is singular  and the quintuple $(a_{0},a_{1},a_{2},a_{3},d)$ is not one of the following four quintuples:
\begin{equation}\label{eq:four-quintuples}
(1,2,3,5,10), (1,3,5,7,15), (1,3,5,8,16),  (2,3,5,9,18).
\end{equation}
Then $S_d$ admits an orbifold K\"ahler--Einstein metric.
\end{theorem}

Its proof uses the criterion given by the $\alpha$-invariant
(for the definition, see \cite[Definition~1.2]{CheltsovShramovUMN}) of the surface $S_d$, see \cite{Ti87,Nadel,DeKo01}.
It says that the surface $S_d$ admits an (orbifold) K\"ahler--Einstein metric if the inequality
\begin{equation}
\label{equation:old-alpha}
\alpha\big(S_d\big)>\frac{2}{3}
\end{equation}
holds, where $\alpha(S_d)$ is the $\alpha$-invariant of the surface $S_d$.
Indeed, Johnson and Koll\'ar verified~\eqref{equation:old-alpha} in the case when
$I=1$, the surface $S_d$ is singular, and the quintuple $(a_{0},a_{1},a_{2},a_{3},d)$ is not one of the four
exceptions~\eqref{eq:four-quintuples}.
Two of the four remaining cases~\eqref{eq:four-quintuples} have been treated in \cite{Ara02} by Araujo, who proved the following result:

\begin{theorem}[{\cite[Theorem~4.1]{Ara02}}]
\label{theorem:Araujo}
In the following two cases:
\begin{itemize}
\item $(a_{0},a_{1},a_{2},a_{3},d)= (1,2,3,5,10)$,
\item $(a_{0},a_{1},a_{2},a_{3},d)=(1,3,5,7,15)$ and the equation of $S_d$ contains~$yzt$,
\end{itemize}
the inequality $\alpha(S_d)>\frac{2}{3}$ holds. In particular, $S_d$ admits an orbifold K\"ahler--Einstein metric.
\end{theorem}

The remaining two cases of~\eqref{eq:four-quintuples} have been dealt with in the  paper \cite{CheltsovParkShramovJGA}.
In this paper, we succeeded in estimating  their
$\alpha$-invariants from below by large enough numbers for the criterion~\eqref{equation:old-alpha}.
To be precise, we proved the following result:

\begin{theorem}[{\cite[Theorem~1.10]{CheltsovParkShramovJGA}}]
\label{theorem:old-I-1}
Suppose that $(a_{0},a_{1},a_{2},a_{3},d)=(1,3,5,8,16)$ or~\mbox{$(2,3,5,9,18)$}.
Then $\alpha(S_d)>\frac{2}{3}$. In particular, $S_d$ admits an orbifold K\"ahler--Einstein metric.
\end{theorem}

In particular, if $I=1$, then $S_d$ admits an orbifold K\"ahler--Einstein metric except possibly the case when
$(a_{0},a_{1},a_{2},a_{3},d)=(1,3,5,7,15)$ and the defining equation of the surface $S_d$ does not contain~$yzt$.
Note that in the latter case one has
$$
\alpha(S_d)=\frac{8}{15}<\frac{2}{3}
$$
by \cite[Theorem~1.10]{CheltsovParkShramovJGA}, so that the criterion by the $\alpha$-invariant  could not be applied.

Meanwhile, since 2010 we have witnessed dramatic developments in the study of the Yau--Tian--Donaldson conjecture concerning
the existence of K\"ahler--Einstein metrics on Fano manifolds and stability.
The challenge to the  conjecture has been heightened  by Chen, Donaldson, Sun and Tian who have completed the proof for the case of Fano manifolds with anticanonical polarisations \cite{CDS15,TianYTD}.
Following this celebrated achievement, useful technologies have been introduced to determine whether given Fano varieties are K\"ahler--Einstein or not, via the theorem of Chen--Donaldson--Sun and Tian.

Recently Fujita and Odaka introduced a new invariant of Fano varieties, which they called $\delta$-invariant (for the definition, see \cite[Definition~1.2]{FujitaOdaka}),
that serves as a strong criterion for uniform K-stability (see \cite{FujitaOdaka}).

\begin{theorem}[{\cite{FujitaOdaka, BJ17}}]
\label{theorem:delta}
Let $X$ be a Fano variety with at most Kawamata log terminal singularities.
Then $X$ is uniformly K-stable if and only if $\delta(X)>1$.
\end{theorem}

This powerful tool has been practiced for del Pezzo surfaces in \cite{ParkWon,CRZ18,CZ}.
Around the same time, Li, Tian and Wang proved in \cite{LTW17,LTW19} that the result of Chen, Donaldson, Sun and Tian
also holds for some singular Fano varieties.
In particular, it holds for del Pezzo surfaces with quotient singularities.
Thus, if $\delta(S_d)>1$, then the surface $S_d$ admits an (orbifold) K\"ahler--Einstein metric.
Note that $3\alpha(S_d)\geqslant\delta(S_d)\geqslant \frac{3}{2}\alpha(S_d)$ by \cite[Theorem~A]{BJ17}.

Now we are strongly reinforced by these new technologies, so that we could complete
the study of existence of an (orbifold) K\"ahler--Einstein metric on the surface $S_d$ in the case $I=1$
started by Johnson and Koll\'ar in \cite{JoKo01b}.
In this paper, we prove the following result:

\begin{theorem}
\label{theorem:main-I-1}
Let $S_d$ be a quasi-smooth hypersurface in $\mathbb{P}(1,3,5,7)$ of degree $15$ such that its
defining equation does not contain~$yzt$.
Then $\delta(S_d)\geqslant\frac{6}{5}$.
In particular, the surface $S_d$ admits an orbifold K\"ahler--Einstein metric.
\end{theorem}

\begin{corollary}
\label{corollary:main-I-1}
If $I=1$, then $S_d$ admits an orbifold K\"ahler--Einstein metric.
\end{corollary}

For  $I\geqslant 2$, the problem of existence of an orbifold K\"ahler--Einstein metric on the surface $S_d$
was first studied by Boyer, Galicki and Nakamaye  in \cite{BoGaNa03}.
In this case, there is no reasonable classification similar to that obtained by Johnson and Koll\'ar in \cite{JoKo01b};
note however that \cite{Erik} presents an algorithm that produces
the (infinite) list of all possibilities for the quintuple $(a_{0},a_{1},a_{2},a_{3},d)$ for every fixed $I\geqslant 2$.
In \cite{BoGaNa03,CheltsovParkShramovJGA,ChSh09c},
the existence of an orbifold K\"ahler--Einstein metric has been proved for (infinitely) many surfaces $S_d$ with $I\geqslant 2$.
However, in the following six cases their method did not work:
\begin{enumerate}
\item $(a_{0},a_{1},a_{2},a_{3},d)=(2,3,4,5,12)$ and the equation of $S_d$ does not contain $yzt$;
\item $(a_{0},a_{1},a_{2},a_{3},d)=(7,10,15,19,45)$;
\item $(a_{0},a_{1},a_{2},a_{3},d)=(7,18,27,37,81)$;
\item $(a_{0},a_{1},a_{2},a_{3},d)=(7,15,19,32,64)$;
\item $(a_{0},a_{1},a_{2},a_{3},d)=(7,19,25,41,82)$;
\item $(a_{0},a_{1},a_{2},a_{3},d)=(7,26,39,55,117)$.
\end{enumerate}
In this paper, we use $\delta$-invariants to show that $S_d$ is K\"ahler--Einstein in these six cases as well:

\begin{theorem}
\label{theorem:main}
Suppose that $(a_{0},a_{1},a_{2},a_{3},d)$ is one of the six quintuples listed above.
Then~\mbox{$\delta(S_d)\geqslant\frac{65}{64}$}.
In particular, the surface $S_d$ admits an orbifold K\"ahler--Einstein metric.
\end{theorem}

According to the similarity of the proofs,
we handle the seven types of del Pezzo surfaces in Theorems~\ref{theorem:main-I-1} and \ref{theorem:main} into
three cases as follows:
\begin{itemize}
\item[] \textbf{Case A.} $(a_{0},a_{1},a_{2},a_{3},d)=(1,3,5,7,15)$ and the equation of $S_d$ does not contain~$yzt$;
\item[]  \hspace{43pt} $(a_{0},a_{1},a_{2},a_{3},d)=(2,3,4,5,12)$ and the equation of $S_d$ does not contain $yzt$;

\item[]  \textbf{Case B.} $(a_{0},a_{1},a_{2},a_{3},d)=(7,15,19,32,64)$;
\item[]  \hspace{43pt} $(a_{0},a_{1},a_{2},a_{3},d)=(7,19,25,41,82)$;

\item[] \textbf{Case C.} $(a_{0},a_{1},a_{2},a_{3},d)=(7,10,15,19,45)$;
\item[]  \hspace{43pt} $(a_{0},a_{1},a_{2},a_{3},d)=(7,18,27,37,81)$;
\item[]   \hspace{43pt} $(a_{0},a_{1},a_{2},a_{3},d)=(7,26,39,55,117)$.
\end{itemize}

We will handle each of these cases separately in Sections~\ref{section:case-a}, \ref{section:case-b} and \ref{section:case-c}, respectively;
see Corollaries~\ref{corollary:1}, \ref{corollary:5} and \ref{corollary:3}.
In Section~\ref{section:basic-tool}, we will present some results that will be used in the proofs of Theorems~\ref{theorem:main-I-1} and \ref{theorem:main}.

Let us briefly explain how we estimate $\delta(S_d)$ in the proofs of Theorems~\ref{theorem:main-I-1} and \ref{theorem:main}.
In our old paper \cite{CheltsovParkShramovJGA}, we developed a technique how to study possible singularities of log pairs
$(S_d,D)$, where $D$ is an effective $\mathbb{Q}$-divisor on the surface $S_d$ such that $D\sim_{\mathbb{Q}} -K_{S_d}$.
This resulted in explicit values of $\alpha(S_d)$ in all considered cases.
To estimate  $\delta(S_d)$, one has to study singularities of similar log pairs with an additional condition: the $\mathbb{Q}$-divisor $D$
has to be of \emph{$k$-basis type} for $k\gg 1$ (for the definition, see \cite[Definition~1.1]{FujitaOdaka}).
By \cite[Lemma 2.2]{FujitaOdaka} (see also Theorem~\ref{theorem:Fujita} below),
this extra condition provides strong upper bounds on multiplicities of the $\mathbb{Q}$-divisor $D$
in various curves on $S_d$.
We use these bounds (for some very particular curves in~$S_d$) together with our original methods developed in \cite{CheltsovParkShramovJGA},
to obtain the required estimates for $\delta(S_d)$.
This approach was first used in \cite{CRZ18} to estimate $\delta$-invariants of the so-called asymptotically del Pezzo surfaces.
Nevertheless, in our case we have an additional difficulty arising from the singularities of the surface $S_d$,
while all surfaces considered in \cite{CRZ18} are smooth.

It would be interesting to study the problem of existence of an orbifold K\"ahler--Einstein metric on $S_d$
in the remaining cases.
In some of these cases, the del Pezzo surface $S_d$ is indeed not K\"ahler--Einstein.
For instance, the surface $S_d$ does not admit an orbifold K\"ahler--Einstein metric in the case when $I>3a_0$.
This~follows from the obstruction found by Gauntlett, Martelli, Sparks, and Yau~\cite{GMSY}.
On the other hand, we expect the following to be true:

\begin{conjecture}
\label{conjecture:I-2}
If $I=2$ or $I=3$, then $S_d$ admits an orbifold K\"ahler--Einstein metric.
\end{conjecture}

We believe that this conjecture can be proved using a similar approach to the one we use in the proofs of  Theorems~\ref{theorem:main-I-1} and \ref{theorem:main}.

\medskip

\textbf{Acknowledgements.}
We started to work on this paper during our stay at the Erwin Schr\"odinger International Institute for Mathematics and Physics in Vienna in August 2018.
We are grateful to the institute for excellent working conditions.
Ivan Cheltsov was supported by the Royal Society grant No. 	IES\textbackslash R1\textbackslash 180205
and the Russian Academic Excellence Project~\mbox{5-100}.
Jihun Park was supported by IBS-R003-D1, Institute for Basic Science in Korea.
Constantin Shramov was supported by
the Russian Academic Excellence Project~\mbox{5-100} and Young Russian Mathematics award.

\section{Basic tools}
\label{section:basic-tool}

Let $S$ be a surface with at most cyclic quotient singularities,
let $C$ be an irreducible reduced curve on $S$,
let $P$ be a point of the curve $C$, and let $D$ be an effective $\mathbb{R}$-divisor on the surface~$S$.
In this section, we present a few of  well-known (local and global) results
that will be used in the proof of Theorem~\ref{theorem:main}.
We start with

\begin{lemma}[\cite{Ko97}]
\label{lemma:mult-1}
Suppose that $P$ is a smooth point of the surface $S$,
and the singularities of the log pair $(S,D)$ are not log canonical at $P$.
Then $\mathrm{mult}_P(D)>1$.
\end{lemma}

This immediately implies

\begin{corollary}
\label{corollary:mult-1}
If $P$ is a smooth point of the surface $S$,
the log pair $(S,D)$ is not log canonical at $P$,
and $C$ is not contained in the support of the divisor $D$,
then $D\cdot C>1$.
\end{corollary}

To state an analogue of this result in the case when $S$ is singular at $P$,
recall that $S$ has a cyclic quotient singularity of type $\frac{1}{n}(a,b)$ at the point $P$, where $a$ and $b$ are coprime positive integers that are also coprime to $n$.
Thus, if $n=1$, then $P$ is a smooth point of the surface $S$.
For~$n>1$, Corollary~\ref{corollary:mult-1} can be generalized as follows:

\begin{lemma}
\label{lemma:mult-1-n}
Suppose that the log pair $(S,D)$ is not log canonical at~$P$,
and $C$ is not contained in the support of the divisor $D$.
Then $D\cdot C>\frac{1}{n}$.
\end{lemma}

\begin{proof}
This follows from \cite[Lemma~2.2]{CheltsovParkShramovJGA}
and \cite[Lemma~2.3]{CheltsovParkShramovJGA}, cf. \cite{BMO}.
\end{proof}

In general, the curve $C$ may be contained in the support of the divisor $D$. Thus, we write
$$
D=aC+\Delta,
$$
where $a$ is a non-negative real number, and $\Delta$ is an effective $\mathbb{R}$-divisor on $S$ whose support does not contain the curve $C$.
Then we have the following useful result:

\begin{lemma}
\label{lemma:inversion-of-adjunction}
Suppose that $a\leqslant 1$, the surface $S$ is smooth at the point $P$, the curve $C$ is also smooth at $P$, and the log pair $(S,D)$ is not log canonical at $P$.
Then
$$
C\cdot\Delta\geqslant\big(C\cdot\Delta\big)_P>1,
$$
where $\big(C\cdot\Delta\big)_P$ is the local intersection number of $C$ and $\Delta$ at $P$.
\end{lemma}

\begin{proof}
This is a special case of a much more general result, known as the inversion of adjunction (see \cite{Shokurov,Prokhorov}).
\end{proof}

The inversion of adjunction also holds for singular varieties.
In our two-dimensional case, it can be stated as follows:

\begin{lemma}
\label{lemma:inversion-of-adjunction-orbifold}
Suppose that $a\leqslant 1$, the log pair $(S,C)$ is purely log terminal at $P$, and the log pair $(S,D)$ is not log canonical at $P$.
Then
$$
C\cdot\Delta>\frac{1}{n}.
$$
\end{lemma}

\begin{proof}
The required inequality follows from a more general version of the inverse of adjunction (see \cite{Shokurov,Prokhorov}).
See also the proof of \cite[Lemma~2.5]{CheltsovParkShramovJGA}.
\end{proof}

By our assumption, the surface $S$ has a cyclic quotient singularity of type $\frac{1}{n}(a,b)$ at the point~$P$.
Thus, locally near $P$, the surface $S$ is a quotient of $\mathbb{C}^2$ by the group $\mathbb{Z}_n$ that acts on~$\mathbb{C}^2$ as
$$
\big(x,y\big)\mapsto\big(\omega^ax,\omega^by\big),
$$
where $\omega$ is a primitive $n$th root of unity.
We can consider $x$ and $y$ as weighted coordinates around the point $P$.

\begin{remark}
\label{remark:plt}
The pair $(S,C)$ has purely log terminal singularity at $P$ if and only if $C$ is given by $x=0$ or $y=0$ for an appropriate choice of weighted coordinates $x$ and $y$.
This follows from \cite[Theorem~2.1.2]{Prokhorov}, see also \cite[\S~9.6]{Ko97}.
Geometrically, this means that $C$ is smooth at~$P$, and its proper transform on the minimal resolution of singularities
of the singular point~$P$ intersects the tail curve in the chain of exceptional curves.
If $(S,C)$ has purely log terminal singularities, then
$$
\big(K_S+C\big)\cdot C=-2+\sum_{O\in C}\frac{n_O-1}{n_O},
$$
where we assume that $S$ has a cyclic quotient singularity of index $n_O$ at the point $O$.
\end{remark}

Let $f\colon\widetilde{S}\to S$ be the weighted blow up of the point $P$ with $\mathrm{wt}(x)=a$ and $\mathrm{wt}(y)=b$,
and let~$E$ be the exceptional curve of the morphism $f$.
Then $\widetilde{S}$ has at most cyclic quotient singularities, one has~\mbox{$E\cong\mathbb{P}^1$}, and the log pair $(\widetilde{S},E)$  has purely log terminal singularities.
Moreover, the curve $E$ has at most two singular points of the surface $\widetilde{S}$.
One of then is a singular point of type $\frac{1}{a}(n,-b)$,
and another is a singular point of type $\frac{1}{b}(-a,n)$.
Furthermore, we have
$$
K_{\widetilde{S}}\sim_{\mathbb{Q}}f^*\big(K_S\big)+\frac{a+b-n}{n}E.
$$
If the curve $C$ is locally given by $x=0$ near the point $P$, then
$$
\widetilde{C}\sim_{\mathbb{Q}}f^*\big(C\big)-\frac{a}{n}E,
$$
where $\widetilde{C}$ is the proper transform of the curve $C$ on the surface $\widetilde{S}$.
For more properties of weighted blow ups and their defining equations, see \cite[Section~3]{Prokhorov} or \cite{BMO}.

Denote by $\widetilde{D}$ the proper transform of the divisor $D$ via $f$. Then
$$
\widetilde{D}\sim_{\mathbb{R}}f^*\big(D\big)-mE
$$
for some  non-negative rational number $m$. If $C$ is not contained in the support of the divisor~$D$, we can estimate $m$ using
$$
0\leqslant\widetilde{D}\cdot\widetilde{C}=\big(f^*\big(D\big)-mE\big)\cdot\widetilde{C}=D\cdot C-mE\cdot \widetilde{C},
$$
where $D\cdot C$ and $E\cdot \widetilde{C}$ can be computed in every case. Note that
$$
K_{\widetilde{S}}+\widetilde{D}+\Big(m-\frac{a+b-n}{n}\Big)E\sim_{\mathbb{R}}f^*\big(K_S+D\big).
$$
This implies

\begin{proposition}
\label{proposition:lc-invariant-under-blowup}
The log pair $(S,D)$ is log canonical at $P$ if and only if the log pair
$$
\Bigg(\widetilde{S},\widetilde{D}+\Big(m-\frac{a+b-n}{n}\Big)E\Bigg)
$$
is log canonical along the curve $E$.
\end{proposition}

So far, we considered only local properties of the divisor $D$ on the surface $S$.
These properties will be used later to prove Theorem~\ref{theorem:main}.
However, the nature of this theorem is global, so that we will need one global result that is due to Fujita and Odaka.
To state it, we remind the reader of  what the volume $\mathrm{vol}(D)$ of the $\mathbb{R}$-divisor $D$ is.
If $D$ is a Cartier divisor, then its volume is simply the number
$$
\mathrm{vol}(D)=\limsup_{k\in\mathbb{N}}\frac{h^0(\mathcal{O}_S(kD)}{k^2/2!},
$$
where the $\limsup$ can be replaced by  limit (see \cite[Example~11.4.7]{Laz-positivity-in-AG}).
Likewise, if $D$ is a $\mathbb{Q}$-divisor, we can define its volume using the identity
$$
\mathrm{vol}(D)=\frac{\mathrm{vol}\big(\lambda D\big)}{\lambda^2}
$$
for an appropriate positive rational number $\lambda$.
One can show that the volume $\mathrm{vol}(D)$ only depends on the numerical equivalence class of the divisor $D$.
Moreover, the volume function can be continuously extended to $\mathbb{R}$-divisors (see \cite{Laz-positivity-in-AG} for details).

If $D$ is not pseudoeffective, then $\mathrm{vol}(D)=0$.
If $D$ is pseudoeffective, its volume can be computed using its Zariski decomposition \cite{Prokhorov03,BauerKuronyaSzemberg}.
Namely, if $D$ is pseudoeffective, then there exists a nef $\mathbb{R}$-divisor $N$ on the surface $S$ such that
$$
D\sim_{\mathbb{R}} N+\sum_{i=1}^ra_iC_i,
$$
where each $C_i$ is an irreducible curve on $S$ with $N\cdot C_i=0$,
each $a_i$ is a non-negative real number, and the intersection form of the curves $C_1,\ldots,C_r$ is negative definite.
Such decomposition is unique, and it follows from \cite[Corollary~3.2]{BauerKuronyaSzemberg} that
\begin{equation}
\label{equation:volume}
\mathrm{vol}\big(D\big)=\mathrm{vol}\big(N\big)=N^2.
\end{equation}

Recall that $D=aC+\Delta$,
where $a$ is a non-negative real number, and $\Delta$ is an effective divisor whose support does not contain the curve $C$.
Let
$$
\tau=\sup\Big\{x\in\mathbb{R}_{>0}\ \Big|\ D-xC\ \text{is pseudoeffective}\Big\}.
$$
Then $a\leqslant\tau$.
However, to prove Theorem~\ref{theorem:main}, we have to find a better bound for $a$
under an additional assumption that $D$ is an ample  $\mathbb{Q}$-divisor  of $k$-basis type for $k\gg 1$
(for the definition, see \cite[Definition~1.1]{FujitaOdaka} and the proof of Theorem~\ref{theorem:Fujita} below).
One such estimate is given by the following very special case of \cite[Lemma 2.2]{FujitaOdaka}.

\begin{theorem}
\label{theorem:Fujita}
Suppose that $D$ is a big  $\mathbb{Q}$-divisor of $k$-basis type for $k\gg 1$.
Then
$$
a\leqslant\frac{1}{D^2}\int_0^{\tau}\mathrm{vol}\big(D-xC\big)dx+\epsilon_k,
$$
where $\epsilon_k$ is a small constant depending on $k$ such that $\epsilon_k\to 0$ as $k\to \infty$.
\end{theorem}

\begin{proof}
Let us give a sketch of the proof that shows the nature of the required bound.
First, recall from \cite{FujitaOdaka} that being $k$-basis type simply means that
$$
D=\frac{1}{kd_k}\sum_{i=1}^{d_k}\big\{s_i=0\big\},
$$
where $d_k=h^0(S, \mathcal{O}_{S}(kD))$ and $s_1,\ldots,s_{d_k}$ are linearly independent sections in $H^0(S, \mathcal{O}_{S}(kD))$.
Here, we assume that $kD$ is a Cartier divisor and $k\gg 0$.

Let $M$ be a positive rational number such that $M\geqslant\tau$.
We may assume that $kM$ is an integer.
Then there is a filtration of vector spaces
\begin{multline*}
0=H^0\Big(S, \mathcal{O}_{S}(kD-(kM+1)C)\Big)\subseteq H^0\Big(S, \mathcal{O}_{S}(kD-kMC)\Big)\subseteq\\
\subseteq H^0\Big(S, \mathcal{O}_{S}(kD-(kM-1)C)\Big)\subseteq \ldots \subseteq H^0\Big(S, \mathcal{O}_{S}(kD-3C)\Big)\subseteq\\
\subseteq H^0\Big(S, \mathcal{O}_{S}(kD-2C)\Big)\subseteq H^0\Big(S, \mathcal{O}_{S}(kD-C)\Big)\subseteq H^0\Big(S, \mathcal{O}_{S}(kD)\Big).
\end{multline*}
Let $r_i=h^0(S, \mathcal{O}_{S}(kD-iC))$. Then
$$
0=r_{kM+1}\leqslant r_{kM}\leqslant r_{kM-1}\leqslant \ldots\leqslant r_3\leqslant r_2\leqslant r_1\leqslant r_0=d_k.
$$
Since the sections $s_1,\ldots,s_{d_k}$ are linearly independent,
we see that at most $r_1$ of them are contained in
$$
H^0\Big(S, \mathcal{O}_{S}(kD-C\Big).
$$
Among them at most $r_2$ are contained in $H^0(S, \mathcal{O}_{S}(kD-2C))$.
Among them at most $r_3$ are contained in $H^0(S, \mathcal{O}_{S}(kD-3C))$ etc.
Finally, at most $r_{kM}$ sections among $s_1,\ldots,s_{d_k}$ are contained in
$$
H^0\Big(S, \mathcal{O}_{S}(kD-kMC\Big),
$$
and there are no sections in $H^0(\mathcal{O}_{S}(kD-(kM+1)C)=0$. Then
\begin{itemize}
\item at most $r_1$ sections among $s_1,\ldots,s_{d_k}$ vanish at $C$;
\item at most $r_2$ sections among $s_1,\ldots,s_{d_k}$ vanish at $C$ with order $\geqslant 2$;
\item at most $r_3$ sections among $s_1,\ldots,s_{d_k}$ vanish at $C$ with order $\geqslant 3$;
\item $\ldots$
\item at most $r_{kM-1}$ sections among $s_1,\ldots,s_{d_k}$ vanish at $C$ with order $\geqslant kM-1$;
\item at most $r_{kM}$ sections among $s_1,\ldots,s_{d_k}$ vanish at $C$ with order $kM$;
\item no sections among $s_1,\ldots,s_{d_k}$ vanish at $C$ with order $\geqslant kM+1$.
\end{itemize}
This immediately implies that the the order of vanishing of the product  $s_1\cdot s_2\cdot s_3\cdot\ldots\cdot s_{d_n}$ at the curve $C$ is at most
\begin{multline*}
kMr_{kM}+(kM-1)\big(r_{kM-1}-r_{kM}\big)+(kM-2)\big(r_{kM-2}-r_{kM-1}\big)+\ldots\\
\ldots+4\big(r_{4}-r_{5}\big)+3\big(r_{3}-r_{4}\big)+2\big(r_{2}-r_{3}\big)+\big(r_{1}-r_{2}\big)=\sum_{i=1}^{kM}r_i.
\end{multline*}
Then we have
$$
a\leqslant\frac{r_1+r_2+\ldots+r_{kM-1}+r_{kM}}{kr_0}.
$$
As $k\to\infty$, the right hand side in this inequality converges to
$$
\frac{1}{D^2}\int_0^{\tau}\mathrm{vol}\big(D-xC\big)dx,
$$
which gives the upper bound on $a$.
For a detailed proof, we refer the reader to \cite{FujitaOdaka}.
\end{proof}

\begin{corollary}
\label{corollary:a-simple}
Suppose that $D$ is a big  $\mathbb{Q}$-divisor of $k$-basis type for $k\gg 1$,
and
$$
C\sim_{\mathbb{Q}} \mu D
$$
for some positive rational number $\mu$. Then
$$
a\leqslant\frac{1}{3\mu}+\epsilon_k,
$$
where $\epsilon_k$ is a small constant depending on $k$ such that $\epsilon_k\to 0$ as $k\to \infty$.
\end{corollary}

\begin{proof}
Using Theorem~\ref{theorem:Fujita}, we get
$$
a\leqslant\frac{1}{D^2}\int_0^{\infty}\mathrm{vol}\big(D-\lambda C\big)d\lambda+\epsilon_k,
$$
where $\epsilon_k$ is a small constant depending on $k$ such that $\epsilon_k\to 0$ as $k\to \infty$.
But
$$
\int_0^{\infty}\mathrm{vol}\big(D-\lambda C\big)d\lambda=
\int_0^{\infty}\mathrm{vol}\big((1-\lambda\mu)D\big)d\lambda=
D^2\int_0^{\frac{1}{\mu}}(1-\lambda\mu)^2d\lambda=\frac{D^2}{3\mu}.
$$
This implies the assertion.
\end{proof}

\section{Case A}
\label{section:case-a}
In this section, we consider two types of quasismooth hypersurfaces as follows:
\begin{itemize}
\item $S_{15}$ : a quasismooth hypersurface in $\mathbb{P}(1,3,5,7)$ of degree $15$;
\item $S_{12}$ : a quasismooth hypersurface in $\mathbb{P}(2,3,4,5)$ of degree $12$.
\end{itemize}

 By suitable coordinate changes,
$S_{15}$ may be assumed to be  given by
\begin{multline*}
z^{3}+y^{5}+xt^{2}+b_1yzt+b_2xy^{3}z+b_3x^{2}yz^{2}+b_4x^{2}y^{2}t+\\
+b_5x^{3}zt+b_6x^{3}y^{4}+b_{7}x^{4}y^{2}z+b_{8}x^{5}z^{2}+b_{9}x^{5}yt+b_{10}x^{6}y^{3}+\\
+b_{11}x^{7}yz+b_{12}x^{8}t+b_{13}x^{9}y^{2}+b_{14}x^{10}z+b_{15}x^{12}y+b_{16}x^{15}=0
\end{multline*}
and $S_{12}$ by
$$
z(z-x^2)(z-\epsilon x^2)+y^4 +xt^2+b_1yzt+b_2xy^2z+b_3x^2yt+b_4x^3y^2=0,
$$
where $\epsilon$ ($\epsilon\ne 0$ and $\epsilon\ne 1$), $b_1$,
$b_2$, $b_3$, $b_4$, $b_5$, $b_6$, $b_7$, $b_8$, $b_9$, $b_{10}$,
$b_{11}$, $b_{12}$, $b_{13}$, $b_{14}$, $b_{15}$ and $b_{16}$  are constants. Note that the surface $S_{15}$ has the only singular point at $O_t=[0:0:0:1]$. Meanwhile, $S_{12}$ has exactly four singular points at $O_x=[1:0:0:0]$, $O_t=[0:0:0:1]$, $Q_1=[1:0:1:0]$ and $Q_2=[1:0:\epsilon:0]$.

In the sequel, we use $S$ for the surfaces $S_{15}$ and $S_{12}$ if properties or conditions are satisfies by both the surfaces.

Denote by $C_x$ the curve in $S$  cut out by the equation $x=0$.
Then the curve $C_{x}$ is reduced and irreducible in both the cases. It is easy to check
$$
\mathrm{lct}\big(S_{15}, C_{x}\big)=\left\{%
\aligned
&1\ \text{ if $a_1\ne 0$},\\%
&\frac{8}{15}\ \text{ if $a_1=0$},\\%
\endaligned\right.%
$$
$$
\mathrm{lct}\big(S_{12}, C_x\big)=\left\{%
\aligned
&1\ \text{ if $a_1\ne0$},\\%
&\frac{7}{12}\ \text{ if $a_1=0$},\\%
\endaligned\right.%
$$
where $\mathrm{lct}\big(S, C_{x}\big)$ is the log canonical threshold of $C_x$ on $S$.
Moreover, one has $\alpha(S)=\mathrm{lct}(S,C_{x})$ by \cite[Theorem~1.10]{CheltsovParkShramovJGA}.
Thus, since $\delta(S)\geqslant\frac{3}{2}\alpha(S)$, we obtain

\begin{corollary}
\label{corollary:1-delta-easy}
If $b_1\ne 0$, then $\delta(S)\geqslant\frac{3}{2}$.
\end{corollary}

From now on, we suppose that $b_1=0$.
\begin{proposition}
\label{proposition:1}
 Let $D$ be an effective $\mathbb{Q}$-divisor on $S$ such that
$$
D\sim_{\mathbb{Q}} -K_{S}.
$$
Write $D=aC_x+\Delta$, where $a$ is a non-negative number, and $\Delta$ is an effective $\mathbb{Q}$-divisor on the surface $S$ whose support does not contain the curve $C_x$.
Suppose also that~$a\leqslant\frac{8}{21}$.
Then the log pair $(S,\frac{6}{5}D)$ is log canonical.
\end{proposition}

\begin{corollary}
\label{corollary:1}
One has $\delta(S)\geqslant\frac{6}{5}$.
\end{corollary}

\begin{proof}
Let $D$ be a $\mathbb{Q}$-divisor of $k$-basis type divisor on $S$ with $k\gg 0$.
Write $D=aC_x+\Delta$, where~$a$ is a non-negative number, and $\Delta$ is an effective $\mathbb{Q}$-divisor on the surface $S$ whose support does not contain the curve $C_x$.
By Corollary~\ref{corollary:a-simple}, we have $a\leqslant\frac{8}{21}$ for $k\gg 0$.
Thus, the log pair $(S,\frac{6}{5}D)$ is log canonical for $k\gg 0$ by Proposition~\ref{proposition:1}.
This implies that $\delta(S)\geqslant\frac{6}{5}$ by Corollary~\ref{corollary:1-delta-easy}.
\end{proof}

To prove Proposition~\ref{proposition:1}, we
fix an effective $\mathbb{Q}$-divisor $D$ on the surface $S$ such that
$$
D\sim_{\mathbb{Q}} -K_{S}.
$$
Write $D=aC_x+\Delta$, where $a$ is a non-negative number, and $\Delta$ is an effective $\mathbb{Q}$-divisor on the surface $S$ whose support does not contain the curve $C_x$.
Suppose also that~$a\leqslant \frac{8}{21}$.
Let us show that the log pair $(S,\frac{6}{5}D)$ is log canonical.

\begin{lemma}
\label{lemma:1-smooth-points}
The log pair $(S,\frac{6}{5}D)$ is log canonical outside $C_x$.
\end{lemma}

\begin{proof}
The required assertion follows from \cite[Lemma~2.7]{CheltsovParkShramovJGA}.
For convenience of the reader, let us give the detailed proof here.
Let $P$ be a point in $S\setminus C_x$. Since $P\not\in C_x$, there are complex numbers $c_1$ and $c_2$ such that $P$ satisfies the following system of equations:$$
\left\{%
\aligned
&z+c_1x^5=0\\%
&y+c_2x^3=0  \ \ \ \mbox{ for $S_{15}$;}\\%
\endaligned\right.%
$$
$$
\left\{%
\aligned
&y^2+c_1x^3=0\\%
&z+c_2x^2=0 \ \ \ \mbox{ for $S_{12}$.}\\%
\endaligned\right.%
$$

Let $\mathcal{P}$ be the pencil of curves that is given by
$$
\nu\big(z+c_1x^5\big)+\mu\big(yx^2+c_2x^5\big)=0 \ \ \  \mbox{ on  $S_{15}$},
$$
$$
\nu\big(y^2+c_1x^3\big)+\mu \big(zx+c_2x^3\big)=0  \ \ \  \mbox{ on   $S_{12}$}
$$
for $[\nu:\mu]\in\mathbb{P}^1$. Then the base locus of the pencil $\mathcal{P}$ consists of finitely many points.
Moreover, by construction, the point $P$ is one of them.
Let $C$  be a general curve in $\mathcal{P}$.
Then
$$C\cdot D\leqslant \frac{5}{6},$$ so that $(S,\frac{6}{5}D)$ is log canonical at~$P$ by Corollary~\ref{corollary:mult-1} if
$P$ is a smooth point of the surface $S$. This verifies the statement for $S_{15}$.

For $S_{12}$, we suppose that $(S_{12},\frac{6}{5}D)$ is not log canonical at $P$. Then  $P$ must be one of the points $O_x$, $Q_1$,$Q_2$. Observe that the point $P$ belongs to the curve  $C_y$ cut by $y=0$.
Moreover, the curve $C_y$ is irreducible and the log pair $(S_{12},\frac{6}{5}\cdot\frac{2}{3}C_y)$ is log canonical.
Thus, it follows from \cite[Remark~2.22]{CheltsovShramovUMN} that there exists
an effective $\mathbb{Q}$-divisor $D^\prime$ on the surface $S_{12}$ such that
$$
D^\prime\sim_{\mathbb{Q}} -K_{S_{12}},
$$
the log pair $(S_{12},\frac{6}{5}D^\prime)$ is not log canonical at the point $P$,
and the support of the divisor $D^\prime$ does not contain the curve $C_y$.
However,
$$D^\prime\cdot C_y=\frac{6}{10},$$
which is impossible by Lemma~\ref{lemma:mult-1-n} since $(S_{12},\frac{6}{5}D^\prime)$ is not log canonical at the point $P$.
This completes the proof for $S_{12}$.
\end{proof}

\begin{lemma}
\label{lemma:1-smooth-points-C-x}
 The log pair $(S,\frac{6}{5}D)$ is log canonical at a point in $C_x\setminus \{O_t\}$.
\end{lemma}

\begin{proof}
Let $P$ be a point in $C_x\setminus \{O_t\}$. Observe that $P$ is a smooth point of the surface $S$, and~$C_x$ is smooth at the point~$P$.
Note also that $\frac{6}{5}a<1$.
Thus, we can apply Lemma~\ref{lemma:inversion-of-adjunction} to~\mbox{$(S,\frac{6}{5}D)$} and the curve $C_x$ at the point $P$.
Indeed, since
$$
\Big(C_x\cdot\Delta\Big)_P\leqslant C_x\cdot\Delta =\frac{1-a}{7}\leqslant \frac{5}{6}  \ \ \  \mbox{ on  $S_{15}$},
$$
$$
\Big(C_x\cdot\Delta\Big)_P\leqslant C_x\cdot\Delta =\frac{1-2a}{5}\leqslant \frac{5}{6}  \ \ \  \mbox{ on  $S_{12}$},
$$
the log pair $(S,\frac{6}{5}D)$ must be log canonical at $P$.
\end{proof}

Note that $S_{15}$ (resp. $S_{12}$) has singularity of type $\frac{1}{7}(3,5)$  (resp. $\frac{1}{5}(3,4)$) at the point $O_t$.
In the chart defined by $t=1$, the surface $S_{15}$ is given by
\begin{multline*}
z^{3}+y^{5}+x+b_2xy^{3}z+b_3x^{2}yz^{2}+b_4x^{2}y^{2}+\\
+b_5x^{3}z+b_6x^{3}y^{4}+b_{7}x^{4}y^{2}z+b_{8}x^{5}z^{2}+b_{9}x^{5}y+b_{10}x^{6}y^{3}+\\
+b_{11}x^{7}yz+b_{12}x^{8}+b_{13}x^{9}y^{2}+b_{14}x^{10}z+b_{15}x^{12}y+b_{16}x^{15}=0,
\end{multline*}
and $S_{12}$  by
$$
z(z-x^2)(z-\epsilon x^2)+y^4 +x+a_1yz+a_2xy^2z+a_3x^2y+a_4x^3y^2=0.
$$
Thus, in a neighborhood of the point $O_t$, we may regard $y$ and $z$ as local weighted coordinates with $\mathrm{wt}(y)=3$ and $\mathrm{wt}(z)=5$ for $S_{15}$ and with $\mathrm{wt}(y)=3$ and $\mathrm{wt}(z)=4$ for $S_{12}$.

Let $f\colon\widetilde{S}\to S$ be the weighted blow up at the singular point $O_t$ with weights $\mathrm{wt}(y)=3$,
$\mathrm{wt}(z)=5$ for $S_{15}$ and with  weights $\mathrm{wt}(y)=3$, $\mathrm{wt}(z)=4$ for $S_{12}$.
Denote by $E$ the exceptional curve of the blow up $f$.
Then
$$
K_{\widetilde{S}_{15}}\sim_{\mathbb{Q}} f^{*}\big(K_{S_{15}}\big)+\frac{1}{7}E;
$$
$$
K_{\widetilde{S}_{12}}\sim_{\mathbb{Q}} f^{*}\big(K_{S_{12}}\big)+\frac{2}{5}E.
$$
The surface $S$ has two singular points in $E$.
One  is a point of type $\frac{1}{3}(1,1)$,
and the other is a singular point of type $\frac{1}{5}(1,1)$ on $\widetilde{S}_{15}$ ( $\frac{1}{4}(1,1)$ on $\widetilde{S}_{12}$).
Denote the former by $O_3$ and the latter  by $O$.
Observe that
$$E^2=-\frac{7}{15}  \ \ \  \mbox{ on  $\widetilde{S}_{15}$} ;$$
$$E^2=-\frac{5}{12}  \ \ \  \mbox{ on  $\widetilde{S}_{12}$} ;$$
and $E\cong\mathbb{P}^1$.

Let $\widetilde{C}_x$ be the proper transform of the curve $C_x$ on the surface $\widetilde{S}$.
Then
$$
\widetilde{C}_x\sim_{\mathbb{Q}} f^{*}\big(C_x\big)-cE \ \ \  \mbox{ for  $S_{15}$},
$$
where $c=\frac{15}{7}$ for  $S_{15}$ and $c=\frac{12}{5}$ for  $S_{12}$,
and the intersection $E\cap\widetilde{C}_x$ consists of a single point,
which is different from $O_3$ and $O$.
Note that the curves $E$ and $\widetilde{C}_x$ intersect transversally at the point $E\cap\widetilde{C}_x$.

Denote by $\widetilde{\Delta}$ be the proper transform of the $\mathbb{Q}$-divisor $\Delta$ on the surface $\widetilde{S}$.
Then
$$
\widetilde{\Delta}\sim_{\mathbb{Q}} f^{*}\big(\Delta\big)-mE
$$
for some non-negative rational number $m$. To estimate it, observe that
$$
0\leqslant\widetilde{C}_x\cdot\widetilde{\Delta}=\Big(f^{*}\big(C_x\big)-cE\Big)\cdot\Big(f^{*}\big(\Delta\big)-mE\Big)=C_x\cdot\Delta-m=C_x\cdot (D-aC_x)-m,
$$
so that $m\leqslant\frac{1-a}{7}$ for $S_{15}$ and $m\leqslant\frac{1-2a}{5}$ for $S_{12}$. Now we are ready to prove

\begin{lemma}
\label{lemma:1-P-O-t}
The log pair $(S,\frac{6}{5}D)$ is log canonical at~$O_t$.
\end{lemma}

\begin{proof}
Suppose that the log pair $(S,\frac{6}{5}D)$ is not log canonical at~$O_t$. Let us seek for a contradiction.
Let $\lambda=\frac{6}{5}$. Then
$$
K_{\widetilde{S}}+\lambda a\widetilde{C}_x+\lambda\widetilde{\Delta}+\mu E\sim_{\mathbb{Q}}f^*\big(K_S+\lambda D\big),
$$
where $$\mu=\frac{15\lambda a}{7}+\lambda m-\frac{1}{7} \ \ \  \mbox{ for  $S_{15}$},$$
$$\mu=\frac{12\lambda a}{5}+\lambda m-\frac{2}{5} \ \ \  \mbox{ for  $S_{12}$}.$$
Thus, the log pair
\begin{equation}
\label{equation:1-log-pull-back}
\Big(\widetilde{S},\lambda a\widetilde{C}_x+\lambda\widetilde{\Delta}+\mu E\Big)
\end{equation}
is not log canonical at some point $Q\in E$.
Note that
$
\mu\leqslant 1
$
because  $m\leqslant\frac{1-a}{7}$ (or  $m\leqslant\frac{1-2a}{5}$)  and~\mbox{$a\leqslant\frac{8}{21}$}.

We first apply Lemmas~\ref{lemma:inversion-of-adjunction} or~\ref{lemma:inversion-of-adjunction-orbifold}  to  \eqref{equation:1-log-pull-back} and the curve $E$ at the point $Q$.
Indeed,
\[E\cdot\widetilde{\Delta}=E\cdot (f^{*}\big(\Delta\big)-mE)=-mE^2=\left\{%
\aligned
&\frac{7m}{15}\leqslant  \frac{1-a}{15} \leqslant \frac{1}{6}\ \ \ \text{ on $\widetilde{S}_{15}$},\\%
&\frac{5m}{12} \leqslant  \frac{1-2a}{12} \leqslant \frac{5}{24} \ \ \  \text{ on $\widetilde{S}_{12}$}.\\%
\endaligned\right.%
\]
This shows that $Q$ must be the intersection point of $E$ and $\widetilde{C}_x$.

Applying Lemma~\ref{lemma:inversion-of-adjunction} again, we see that
\[\frac{5}{6}=\frac{1}{\lambda}<\Big(a\widetilde{C}_x+\widetilde{\Delta}\Big)\cdot E= a+\widetilde{\Delta}\cdot E =\left\{%
\aligned
& a+\frac{7m}{15}\leqslant  a+ \frac{1-a}{15} \ \ \ \text{ on $\widetilde{S}_{15}$},\\%
& a+\frac{5m}{12} \leqslant  a+ \frac{1-2a}{12} \ \ \  \text{ on $\widetilde{S}_{12}$}.\\%
\endaligned\right.%
\]
However, these inequalities  contradict our assumption $a\leqslant\frac{8}{21}$. Therefore, the log pair $(S,\frac{6}{5}D)$ is log canonical at~$O_t$.
\end{proof}

Proposition~\ref{proposition:1} is completely verified.

\section{Case B}
\label{section:case-b}
The way to evaluate $\delta$-invariants for Case B is almost same as that of Case A. In spite of this, we write the proof for the readers' convenience.

In this section, we consider the following two types of quasismooth hypersurfaces:
\begin{itemize}
\item $S_{64}$ : a quasismooth hypersurface in $\mathbb{P}(7,15,19,32)$ of degree $64$;
\item $S_{82}$ : a quasismooth hypersurface in $\mathbb{P}(7,19,25,41)$ of degree $82$.
\end{itemize}

As in the previous section, we use $S$ for the surfaces $S_{64}$ and $S_{82}$ if properties or conditions are satisfies by both the surfaces.

We may assume that the surface $S_{64}$
 is  given by  the equation
$$
t^{2}+y^{3}z+xz^{3}+x^{7}y=0
$$
in  $\mathbb{P}(7,15,19,32)$
and  $S_{82}$
 by the equation
$$
t^2+y^3z+xz^3+x^9y=0
$$
in $\mathbb{P}(7,19,25,41)$.
The surface  $S$ is singular at the points $O_x=[1:0:0:0]$, $O_y=[0:1:0:0]$ and $O_z=[0:0:1:0]$, and is smooth away from them.
Moreover, the surface $S_{64}$ (resp. $S_{82}$) has quotient singularity of types $\frac{1}{7}(5,4)$,
$\frac{1}{15}(7,2)$, $\frac{1}{19}(2,3)$ (resp. $\frac{1}{7}(2,3)$,
$\frac{1}{19}(7,3)$, $\frac{1}{25}(2,3)$) at the points  $O_x$, $O_y$, $O_z$, respectively.

Let $C_x$ be the curve in $S$ cut out by $x=0$
and~$C_y$ by $y=0$.
Then both the curves $C_x$ and~$C_y$ are irreducible.
We have
$$
\frac{35}{54}=\mathrm{lct}\Big(S_{64}, \frac{9}{7}C_x\Big)<\mathrm{lct}\Big(S_{64}, \frac{9}{15}C_y\Big)=\frac{25}{18},
$$
$$
\frac{7}{12}=\mathrm{lct}\Big(S_{82},\frac{10}{7}C_x\Big)<
\mathrm{lct}\Big(S_{82},\frac{10}{19}C_y\Big)=\frac{19}{12},%
$$
which imply   $\alpha(S_{64})\leqslant \frac{35}{54}$ and $\alpha(S_{82})\leqslant \frac{7}{12}$.
In fact, we have $\alpha(S_{64})=\frac{35}{54}$ and $\alpha(S_{82})=\frac{7}{12}$ by \cite[Theorem~1.10]{CheltsovParkShramovJGA}.

\begin{proposition}
\label{proposition:5}
Let $D$ be an effective $\mathbb{Q}$-divisor on $S$ such that
$$
D\sim_{\mathbb{Q}} -K_{S}.
$$
Write $D=aC_x+\Delta$, where $a$ is a non-negative number, and $\Delta$ is an effective $\mathbb{Q}$-divisor on the surface $S$ whose support does not contain the curve $C_x$.
Suppose also that~$a\leqslant\frac{1}{2}$.
Then the log pair $(S,\frac{19}{18}D)$ is log canonical.
\end{proposition}
\begin{proof}
Suppose also that~$a\leqslant\frac{1}{2}$.

We first consider a point $P$ that lies neither on $C_x$ nor on $C_y$.
Observe that $P$ is a smooth point of the surface $S$.
Since $P\not\in C_x$, there are complex numbers $c_1$ and $c_2$ such that $P$ satisfies the following system of equations:
$$
\left\{%
\aligned
&y^7+c_1x^{15}=0\\%
&y^{2}z+c_2x^7=0 \ \ \ \mbox{ for $S_{64}$;}\,\\%
\endaligned\right.%
$$
$$
\left\{%
\aligned
&y^4+c_1x^5t=0\\%
&y^3+c_2xz^2=0 \ \ \ \mbox{ for $S_{82}$.}\\%
\endaligned\right.%
$$

Moreover, since $P\not\in C_y$, we have $c_1\ne 0$.
Let $\mathcal{P}$ be the pencil given by
$$
\nu\big(y^7+c_1x^{15}\big)+\mu x^8\big(y^{2}z+c_2x^7\big)=0 \ \ \ \mbox{ on  $S_{64}$;}
$$
$$
\nu\big(y^4+c_1x^5t\big)+\mu y\big(y^3+c_2xz^2\big)=0 \ \ \ \mbox{ on  $S_{82}$}
$$
for $[\nu:\mu]\in\mathbb{P}^1$.
The base locus of the pencil $\mathcal{P}$ consists of finitely many points.
Furthermore, by construction, the point $P$ is one of them.
Let $C$  be a general curve in $\mathcal{P}$. Then $$\mathrm{mult}_P(D)\leqslant C\cdot D\leqslant\frac{18}{19}.$$
It immediately follows from Corollary~\ref{corollary:mult-1}  that the log pair $(S,\frac{19}{18}D)$ is log canonical outside  $C_x$ and $C_y$.

We next consider a point $P$ on $C_x$ different from $O_z$.
Since $a\leqslant\frac{1}{2}$, we apply Lemmas~\ref{lemma:inversion-of-adjunction} and ~\ref{lemma:inversion-of-adjunction-orbifold} to the log pair $(S,\frac{18}{19}aC_x+\frac{18}{19}\Delta)$. Indeed, since
$$
\Big(C_x\cdot\Delta\Big)_P\leqslant C_x\cdot\Delta =\frac{18-14a}{285}\leqslant \frac{6}{95}  \ \ \  \mbox{ on  $S_{64}$},
$$
$$
\Big(C_x\cdot\Delta\Big)_P\leqslant C_x\cdot\Delta =\frac{20-14a}{475}\leqslant \frac{18}{19\cdot19}  \ \ \  \mbox{ on  $S_{82}$},
$$
the log pair $(S,\frac{19}{18}D)$ must be log canonical at $P$.

We now let $P$ be a point on $C_y$ different from $O_z$.
Suppose that the log pair $(S,\frac{19}{18}D)$
is not log canonical at the point $P$.
Recall that $(S_{64},\frac{19}{18}\cdot\frac{9}{15}C_y)$ and $(S_{82},\frac{19}{18}\cdot\frac{10}{19}C_y)$ are log canonical, and the curve $C_y$ is irreducible.
Thus, it follows from \cite[Remark~2.22]{CheltsovShramovUMN} that there exists
an effective $\mathbb{Q}$-divisor $D^\prime$ on the surface $S$ such that
$$
D^\prime\sim_{\mathbb{Q}} -K_{S},
$$
the log pair $(S,\frac{19}{18}D^\prime)$
is not log canonical at the point $P$
and the support of the divisor $D^\prime$ does not contain the curve $C_y$.
Observe
\[C_y\cdot D^\prime=\left\{\aligned
&\frac{18}{19\cdot 7} \ \ \ \text{ on $S_{64}$}\\%
&\frac{4}{35} \ \ \  \text{ on $S_{82}$}\\%
\endaligned\right\}\leqslant \frac{18}{19\cdot 7}.\]
This implies that
the log pair $(S,\frac{19}{18}D^\prime)$ is log canonical at the point $P$.
This contradicts our assumption.  Thus, we see that  $(S,\frac{19}{18}D)$ is log canonical away from $O_z$.
Hence, to complete the proof of Proposition~\ref{proposition:5}, we have to show that $(S,\frac{19}{18}D)$ is log canonical at the point $O_z$.

Recall that $S_{64}$ (resp. $S_{82}$) has singularity of type $\frac{1}{19}(2,3)$ (resp. $\frac{1}{25}(2,3)$) at the point $O_z$.
In the chart $z=1$, the surface $S_{64}$ is given by
$$
t^{2}+y^{3}+x+x^{7}y=0
$$
and $S_{82}$ by
\[
t^2+y^3+x+x^9y=0.
\]
In a neighborhoods of the point $O_z$, we can consider $y$ and $t$ as local weighted coordinates such that $\mathrm{wt}(y)=2$ and $\mathrm{wt}(t)=3$.

Let $f\colon\widetilde{S}\to S$ be the weighted blow up at the singular point $O_z$ with weights $\mathrm{wt}(y)=2$
and~\mbox{$\mathrm{wt}(t)=3$}.
Denote by $E$ the exceptional curve of the blow up $f$.
Then
$$
K_{\widetilde{S}_{64}}\sim_{\mathbb{Q}} f^{*}\big(K_{S_{64}}\big)-\frac{14}{19}E;
$$
$$
K_{\widetilde{S}_{82}}\sim_{\mathbb{Q}} f^{*}\big(K_{S_{82}}\big)-\frac{20}{25}E.
$$
The surface $S$ has two singular points in $E$.
One  is a point of type $\frac{1}{2}(1,1)$
and the other  is  of type $\frac{1}{3}(1,1)$.
Denote the former by $O_2$ and the latter by $O_3$.
Observe $$E^2=-\frac{19}{6}  \ \ \  \mbox{ on  $\widetilde{S}_{64}$};$$
$$E^2=-\frac{25}{6}  \ \ \  \mbox{ on  $\widetilde{S}_{82}$}$$
and $E\cong\mathbb{P}^1$.

Let $\widetilde{C}_x$ be the proper transform of the curve $C_x$ on the surface $\widetilde{S}$.
Then
$$
\widetilde{C}_x\sim_{\mathbb{Q}} f^{*}\big(C_x\big)-cE,
$$
where $c=\frac{6}{19}$ for $S_{64}$ and  $c=\frac{6}{25}$ for $S_{82}$,
and the intersection $E\cap\widetilde{C}_x$ consists of a single point different from $O_2$ and $O_3$.
Note that the curves $E$ and $\widetilde{C}_x$ intersect transversally.

Denote by $\widetilde{\Delta}$ be the proper transform of the $\mathbb{Q}$-divisor $\Delta$ on the surface $\widetilde{S}$.
Then
$$
\widetilde{\Delta}\sim_{\mathbb{Q}} f^{*}\big(\Delta\big)-mE
$$
for some non-negative rational number $m$. To estimate it, observe
$$
0\leqslant\widetilde{C}_x\cdot\widetilde{\Delta}=\Big(f^{*}\big(C_x\big)-cE\Big)\cdot\Big(f^{*}\big(\Delta\big)-mE\Big)=C_x\cdot\Delta-m=C_x\cdot (D-aC_x)-m.
$$
This implies
 $m\leqslant\frac{18-14a}{285}$ for $S_{64}$ and $m\leqslant\frac{20-14a}{19\cdot 25}$ for $S_{82}$.

We finally suppose that the log pair $(S,\frac{19}{18}D)$ is not log canonical at~$O_z$.
Let $\lambda=\frac{19}{18}$. Then
$$
K_{\widetilde{S}}+\lambda a\widetilde{C}_x+\lambda\widetilde{\Delta}+\mu E\sim_{\mathbb{Q}}f^*\big(K_S+\lambda D\big),
$$
where $$\mu= \frac{6\lambda a}{19}+\lambda m+\frac{14}{19} \ \ \ \mbox{ for  $S_{64}$}; $$
$$\mu=\frac{6\lambda a}{25}+\lambda m+\frac{20}{25} \ \ \ \mbox{ for  $S_{82}$}. $$
Thus, the log pair
\begin{equation}
\label{equation:5-log-pull-back}
\Big(\widetilde{S},\lambda a\widetilde{C}_x+\lambda\widetilde{\Delta}+\mu E\Big)
\end{equation}
is not log canonical at some point $Q\in E$.

Using $m\leqslant\frac{18-14a}{15\cdot 19}$  for $S_{64}$, $m\leqslant\frac{20-14a}{19\cdot 25}$  for $S_{82}$ and $a\leqslant\frac{1}{2}$, we get
$$
\frac{6\lambda a}{19}+\lambda m+\frac{14}{19}\leqslant \frac{4\lambda a}{15}+\frac{6\lambda}{95}+\frac{14}{19}\leqslant\frac{56\lambda}{285}+\frac{14}{19}=\frac{2422}{2565}<1,
$$
$$
\frac{6\lambda a}{25}+\lambda m+\frac{20}{25}\leqslant \frac{4\lambda a}{19}+\frac{4\lambda}{95}+\frac{4}{5}\leqslant\frac{14\lambda}{95}+\frac{4}{5}=\frac{817}{855}<1.
$$

Since
\[E\cdot\widetilde{\Delta}=E\cdot (f^{*}\big(\Delta\big)-mE)=-mE^2=\left\{%
\aligned
&\frac{19m}{6}\leqslant  \frac{9-7a}{45} \leqslant \frac{6}{19}\ \ \ \text{ on $\widetilde{S}_{64}$},\\%
&\frac{25m}{6} \leqslant  \frac{20-14a}{6\cdot 19} \leqslant \frac{6}{19} \ \ \  \text{ on $\widetilde{S}_{82}$}.\\%
\endaligned\right.%
\]
Lemmas~\ref{lemma:inversion-of-adjunction} and~\ref{lemma:inversion-of-adjunction-orbifold}  imply that $Q$ must be the intersection point of $E$ and $\widetilde{C}_x$.
It then follows from  Lemma~\ref{lemma:inversion-of-adjunction} that
\[\frac{18}{19}=\frac{1}{\lambda}<\Big(a\widetilde{C}_x+\widetilde{\Delta}\Big)\cdot E= a+\widetilde{\Delta}\cdot E =\left\{%
\aligned
& a+\frac{19m}{6}\leqslant  a+ \frac{9-7a}{45} \ \ \ \text{ on $\widetilde{S}_{64}$},\\%
& a+\frac{25m}{6} \leqslant  a+ \frac{20-14a}{6\cdot 19} \ \ \  \text{ on $\widetilde{S}_{82}$}.\\%
\endaligned\right.%
\]
This contradicts our assumption $a\leqslant \frac{1}{2}$. The obtained contradiction completes the proof.
\end{proof}

\begin{corollary}
\label{corollary:5}
One has $\delta(S)\geqslant\frac{19}{18}$.
\end{corollary}

\begin{proof}
See the proof of Corollary~\ref{corollary:1}.
\end{proof}

\section{Case C}
\label{section:case-c}

In this section, we consider the following three types of quasismooth hypersurfaces:
\begin{itemize}
\item $S_{45}$ : a quasismooth hypersurface in $\mathbb{P}(7,10,15,19)$ of degree $45$;
\item $S_{81}$ : a quasismooth hypersurface in $\mathbb{P}(7,18,27,37)$ of degree $81$;
\item $S_{117}$ : a quasismooth hypersurface in $\mathbb{P}(7,26,39,55)$ of degree $117$.
\end{itemize}

As in the previous sections, we use $S$ for all the surfaces $S_{45}$, $S_{81}$, and $S_{117}$ if properties or conditions are satisfies by all the surfaces.

By appropriate coordinate changes, we may assume that
the surface $S_{45}$ is defined  by the equation
$$
z^{3}-y^{3}z+xt^{2}+x^{5}y=0
$$
in $\mathbb{P}(7,10,15,19)$,
the surface $S_{81}$ by
$$
z^{3}-y^{3}z+xt^{2}+x^{9}y=0
$$
in $\mathbb{P}(7,18,27,37)$,
and the surface $S_{117}$ by
$$
z^{3}-y^{3}z+xt^{2}+x^{13}y=0
$$
in $\mathbb{P}(7,26,39,55)$.

The surface  $S$ is singular at the points
$$
O_x=[1:0:0:0], O_y=[0:1:0:0], O_t=[0:0:0:1], Q=[0:1:1:0],
$$
and is smooth away from them.
Moreover, the surface $S_{45}$ (resp. $S_{81}$ and $S_{117}$) has quotient singularity of types $\frac{1}{7}(1,5)$,
$\frac{1}{10}(7,9)$, $\frac{1}{19}(2,3)$, $\frac{1}{5}(1,2)$ (resp. $\frac{1}{7}(3,1)$,
$\frac{1}{18}(7,1)$, $\frac{1}{37}(2,3)$, $\frac{1}{9}(7,1)$ and~\mbox{$\frac{1}{7}(2,3)$},
$\frac{1}{26}(7,3)$, $\frac{1}{55}(2,3)$, $\frac{1}{13}(7,3)$)
at the points  $O_x$, $O_y$, $O_t$, $Q$, respectively.

Let $C_x$ be the curve in $S$ that is cut out by $x=0$.
Then
$$
C_x=L_{xz}+R_x,
$$
where $L_{xz}$ is the curve  given by $x=z=0$ and $R_x$  by $x=z^2-y^3=0$ in the ambient weighted projective space.
These two curves $L_{xz}$ and $R_x$ meets each other at the point $O_t$. Also, we have
\begin{equation}
\label{equation:3-intersection}
\begin{split}
&L_{xz}^2=-\frac{23}{10\cdot 19},\ \ R_x^2=-\frac{8}{5\cdot 19},\ \ L_{xz}\cdot R_{x}=\frac{3}{19} \ \ \ \mbox{ on $S_{45}$;}\\
&L_{xz}^2=-\frac{47}{18\cdot 37},\ \  R_x^2=-\frac{20}{9\cdot 37},\ \
 L_{xz}\cdot R_{x}=\frac{3}{37} \ \ \ \mbox{ on $S_{81}$;}\\
& L_{xz}^2=-\frac{71}{26\cdot 55},\ \ R_x^2=-\frac{32}{13\cdot 55},
\ \ L_{xz}\cdot R_{x}=\frac{3}{55} \ \ \ \mbox{ on $S_{117}$.}
\end{split}
\end{equation}
Note also that the curve $R_x$ is singular at the point $O_t$.

Let $C_y$ be the curve in $S$ cut out by $y=0$.
Then $C_y$ is irreducible and
$$
\frac{35}{54}=\mathrm{lct}\left(S_{45},\frac{6}{7}C_x\right)<\mathrm{lct}\left(S_{45},\frac{6}{10}C_y\right)=\frac{25}{18};%
$$
$$
\frac{35}{72}=\mathrm{lct}\left(S_{81}, \frac{8}{7}C_x\right)
<\mathrm{lct}\left(S_{81}, \frac{8}{18}C_y\right)=\frac{15}{8};%
$$
$$
\frac{7}{18}=\mathrm{lct}\left(S_{117},\frac{10}{7}C_x\right)<
\mathrm{lct}\left(S_{117},\frac{10}{26}C_y\right)=\frac{13}{6}.%
$$
In fact, in these three cases $\alpha(S)$ is given by the numbers  $\frac{35}{54}$, $\frac{35}{72}$, and $\frac{7}{18}$ on the left-hand sides by \cite[Theorem~1.10]{CheltsovParkShramovJGA}.

To estimate $\delta(S)$, we
fix an effective $\mathbb{Q}$-divisor $D$ on the surface $S$ such that
$$
D\sim_{\mathbb{Q}} -K_{S}
$$
and write $D=aL_{xz}+bR_x+\Delta$, where $a$ and $b$ are non-negative numbers, and $\Delta$ is an effective $\mathbb{Q}$-divisor on the surface $S$ whose support does not contain the curves $L_{xz}$ and~$R_x$.

\begin{lemma}\label{lemma:a-b-bound}
If the $\mathbb{Q}$-divisor $D$ is of $k$-basis type with $k\gg 0$, then
\[a\leqslant \left\{%
\aligned
&\frac{2}{5}\\%
&\frac{1}{2}\\%
&\frac{11}{20}\\%
\endaligned\right\}, \ \ \ %
b\leqslant \left\{%
\aligned
&\frac{1}{3}\ \ \ \text{ on $S_{45}$}\\%
&\frac{1}{5}  \ \ \  \text{ on $S_{81}$}\\%
&\frac{12}{25}  \ \ \  \text{ on $S_{117}$}\\%
\endaligned\right\}.%
\]

\end{lemma}

\begin{proof}
Suppose that $D$ is of $k$-basis type  with $k\gg 0$.
Theorem~\ref{theorem:Fujita} implies that
$$
a\leqslant\frac{1}{(-K_S)^2}\int_0^{\infty}\mathrm{vol}\big(-K_S-\lambda L_{xz}\big)d\lambda+\epsilon_k,
$$
where $\epsilon_k$ is a small constant depending on $k$ such that $\epsilon_k\to 0$ as $k\to \infty$.
Since
$$
-K_S-\lambda L_{xz}\sim_{\mathbb{Q}}
\left\{%
\aligned
&\left(\frac{6}{7}-\lambda\right)L_{xz}+\frac{6}{7}R_x\ \ \ \text{ on $S_{45}$}\\%
&\left(\frac{8}{7}-\lambda\right)L_{xz}+\frac{8}{7}R_x \ \ \  \text{ on $S_{81}$}\\%
&\left(\frac{10}{7}-\lambda\right)L_{xz}+
\frac{10}{7}R_x
 \ \ \  \text{ on $S_{117}$}\\%
\endaligned\right.
$$
and $R_x^2<0$, we have $\mathrm{vol}(-K_S-\lambda L_{xz})=0$ for $\lambda\geqslant\frac{6}{7}$ on $S_{45}$,  $\lambda\geqslant\frac{8}{7}$ on $S_{81}$ and $\lambda\geqslant\frac{10}{7}$ on $S_{117}$.
Similarly, using \eqref{equation:3-intersection}, we see that
$$
\left(-K_S-\lambda L_{xz}\right)\cdot R_x=
\left\{%
\aligned
&\left(\left(\frac{6}{7}-\lambda\right)L_{xz}+\frac{6}{7}R_x\right)\cdot R_x=\frac{6-15\lambda}{19\cdot 5}\ \ \ \text{ on $S_{45}$}\\%
&\left(\left(\frac{8}{7}-\lambda\right)L_{xz}+\frac{8}{7}R_x\right)\cdot R_x=\frac{8-27\lambda}{37\cdot 9} \ \ \  \text{ on $S_{81}$}\\%
&\left(\left(\frac{10}{7}-\lambda\right)L_{xz}+\frac{10}{7}R_x\right)\cdot R_x=
\frac{10-39\lambda}{13\cdot 55}
 \ \ \  \text{ on $S_{117}$.}\\%
\endaligned\right.
$$
This shows that the divisor $-K_S-\lambda L_{xz}$  is nef for $\lambda\leqslant\frac{2}{5}$ on $S_{45}$,  $\lambda\leqslant\frac{8}{27}$ on $S_{81}$ and $\lambda\leqslant\frac{10}{39}$ on~$S_{117}$.
Thus, we have
$$
\mathrm{vol}\big(-K_S-\lambda L_{xz}\big)=\big(-K_S-\lambda L_{xz}\big)^2=\left\{%
\aligned
&\frac{54}{665}-\frac{6\lambda}{95}-\frac{23\lambda^2}{190}
\ \ \ \text{ for $\lambda\leqslant\frac{2}{5}$ on $S_{45}$}\\%
&\frac{32}{777}-\frac{8\lambda}{333}-\frac{47\lambda^2}{666} \ \ \  \text{
for $\lambda\leqslant\frac{8}{27}$  on $S_{81}$}\\%
&
\frac{200}{7007}-\frac{12\lambda}{1001}-\frac{36}{715}\lambda^2
 \ \ \  \text{
for $\lambda\leqslant\frac{10}{39}$
 on $S_{117}$.}\\%
\endaligned\right.
$$
To compute $\mathrm{vol}(-K_S-\lambda L_{xz})$ for $\frac{2}{5}<\lambda<\frac{6}{7}$ on $S_{45}$,  $\frac{8}{27}<\lambda<\frac{8}{7}$ on $S_{81}$ and $\frac{10}{39}<\lambda<\frac{10}{7}$  on~$S_{117}$,
 we let
$$
N=\left\{%
\aligned
&\left(\frac{6}{7}-\lambda\right)L_{xz}+\left(\frac{6}{7}-\frac{15\lambda-6}{8}\right)R_x
\ \ \ \text{ for $S_{45}$}\\%
&\left(\frac{8}{7}-\lambda\right)L_{xz}+\left(\frac{8}{7}-\frac{27\lambda-8}{20}\right)R_x \ \ \  \text{
for  $S_{81}$}\\%
&
\left(\frac{10}{7}-\lambda\right)L_{xz}+\left(\frac{10}{7}-\frac{39\lambda-10}{32}\right)R_x
 \ \ \  \text{
for  $S_{117}$.}\\%
\endaligned\right.
$$
Then, using \eqref{equation:3-intersection} again, we  see that $N\cdot R_x=0$ and $N\cdot  L_{xz}\geqslant 0$.
Thus, we conclude that the divisor $N$ is nef on the respective interval for $\lambda$.
This shows that
$$
-K_S-\lambda L_{xz}\sim_{\mathbb{Q}}
\left\{%
\aligned
&N+\frac{15\lambda-6}{8}R_x
\ \ \ \text{ on $S_{45}$}\\%
&N+\frac{27\lambda-8}{20}R_x \ \ \  \text{
on  $S_{81}$}\\%
&
N+\frac{39\lambda-10}{32}R_x
 \ \ \  \text{
on $S_{117}$}\\%
\endaligned\right.
$$
is the Zariski decomposition of the divisor $-K_S-\lambda L_{xz}$.
Hence, we have
$$
\mathrm{vol}\big(-K_S-\lambda L_{xz}\big)=N^2=
\left\{%
\aligned
&\frac{1}{280}(6-7\lambda)^2
\ \ \ \text{ on $S_{45}$}\\%
&\frac{1}{1260}(8-7\lambda)^2 \ \ \  \text{
on  $S_{81}$}\\%
&
\frac{369}{1121120}(10-7\lambda)^2\ \ \  \text{
on $S_{117}$}\\%
\endaligned\right.
$$
by \eqref{equation:volume}. Thus, integrating, we get
$$
a\leqslant\frac{1}{(-K_S)^2}\int_0^{\infty}\mathrm{vol}\big(-K_S-\lambda L_{xz}\big)d\lambda+\epsilon_k=\left\{%
\aligned
&\frac{118}{315}+\epsilon_k
\ \ \ \text{ for $S_{45}$}\\%
&\frac{760}{1701}+\epsilon_k \ \ \  \text{
for  $S_{81}$}\\%
&
\frac{8780}{17199}+\epsilon_k \ \ \  \text{
for $S_{117}$.}\\%
\endaligned\right.
$$
This gives us the asserted bounds for $a$.

Meanwhile, we have
$$
\mathrm{vol}\big(-K_S-\lambda R_x\big)=\big(-K_S-\lambda R_x\big)^2=
\left\{%
\aligned
&\frac{54}{665}-\frac{12\lambda}{95}-\frac{8\lambda^2}{95}
\ \ \ \text{ for $0\leqslant\lambda\leqslant\frac{1}{5}$ on $S_{45}$}\\%
&\frac{32\cdot 21}{9\cdot 37\cdot 49}-\frac{16\lambda}{9\cdot 37}-\frac{20\lambda^2}{9\cdot 37} \ \ \  \text{
for $0\leqslant\lambda\leqslant\frac{4}{27}$  on $S_{81}$}\\%
&
\frac{30}{1001}-\frac{4\lambda}{143}-\frac{32\lambda^2}{715}
 \ \ \  \text{
for $0\leqslant\lambda\leqslant\frac{5}{39}$
 on $S_{117}$.}\\%
\endaligned\right.
$$
since the divisor $-K_S-\lambda R_x$ is nef for the values $\lambda$ in the respective interval.
The Zariski decomposition of the divisor $-K_S-\lambda R_x$ is given by
$$
\left\{%
\aligned
&\underbrace{\Big(\frac{6}{7}-\frac{30\lambda-6}{23}\Big)L_{xz}+\Big(\frac{6}{7}-\lambda\Big)R_x}_{\text{nef $\mathbb{R}$-divisor}}+\frac{30\lambda-6}{23}L_{xz}\ \ \ \text{ for $\frac{1}{5}<\lambda\leqslant\frac{6}{7}$ on $S_{45}$}\\%
&\underbrace{\Big(\frac{8}{7}-\frac{54\lambda-8}{47}\Big)L_{xz}+\Big(\frac{8}{7}-\lambda\Big)R_x}_{\text{nef $\mathbb{R}$-divisor}}+\frac{54\lambda-8}{47}L_{xz}
 \ \ \  \text{
for $\frac{4}{27}<\lambda\leqslant\frac{8}{7}$  on $S_{81}$}\\%
&
 \underbrace{\Big(\frac{10}{7}-
\frac{78\lambda-10}{71}\Big)L_{xz}+\Big(\frac{10}{7}-\lambda\Big)R_x}_{\text{nef $\mathbb{R}$-divisor}}+\frac{78\lambda-10}{71}L_{xz}
 \ \ \  \text{
for $\frac{5}{39}<\lambda\leqslant\frac{10}{7}$
 on $S_{117}$,}\\%
\endaligned\right.
$$
so that we could obtain
$$
\mathrm{vol}\big(-K_S-\lambda R_x\big)=
\left\{%
\aligned
&\Bigg(\Big(\frac{6}{7}-\frac{30\lambda-6}{23}\Big)L_{xz}+\Big(\frac{6}{7}-\lambda\Big)R_x\Bigg)^2=\frac{2}{5\cdot 7\cdot 23}(6-7\lambda)^2\\
&\Bigg(\Big(\frac{8}{7}-\frac{54\lambda-8}{47}\Big)L_{xz}+\Big(\frac{8}{7}-\lambda\Big)R_x\Bigg)^2=\frac{2}{7\cdot 9\cdot 47}(8-7\lambda)^2\\
&
\Bigg(\Big(\frac{10}{7}-\frac{78\lambda-10}{71}\Big)L_{xz}+
\Big(\frac{10}{7}-\lambda\Big)R_x\Bigg)^2=\frac{2}{7\cdot 13\cdot 71}(10-7\lambda)^2.\\
\endaligned\right.
$$
Finally,  $\mathrm{vol}(-K_S-\lambda R_x)=0$ for $\lambda>\frac{6}{7}$ on $S_{45}$,  for $\lambda>\frac{8}{7}$ on $S_{81}$, and for $\lambda>\frac{10}{7}$ on $S_{117}$ since~\mbox{$-K_S-\lambda R_x$} is not pseudoeffective for these values $\lambda$.
Thus, by Theorem~\ref{theorem:Fujita}, we have
$$
b\leqslant\frac{1}{(-K_S)^2}\int_0^{\infty}\mathrm{vol}\big(-K_S-\lambda R_x\big)d\lambda+\varepsilon_k=
\left\{%
\aligned
&\frac{97}{315}+\varepsilon_k
\ \ \ \text{ for $S_{45}$}\\%
&\frac{10709068}{58281363}+\varepsilon_k
 \ \ \  \text{
for  $S_{81}$}\\%
&
\frac{1205}{2457}+\varepsilon_k\ \ \  \text{
for $S_{117}$.}\\%
\endaligned\right.
$$
This yields the required bounds for $b$.
\end{proof}

Now we prove the main assertion in this section.

\begin{proposition}
\label{proposition:3}
If $a$ and $b$ satisfies the bounds in Lemma~\ref{lemma:a-b-bound}
then the log pair $(S,\frac{65}{64}D)$ is log canonical.
\end{proposition}
\begin{proof}

We suppose that $a$ and $b$ satisfies the bounds in Lemma~\ref{lemma:a-b-bound}.

We fist claim that the log pair $(S,\frac{65}{64}D)$ is log canonical outside of $C_x$ and $C_y$.
This immediately follows from the same argument as in the beginning of the proof of Proposition~\ref{proposition:5} with the pencil
$\mathcal{P}$ given by
$$
\nu\big(x^{10}+c_1y^7\big)+\mu y^4\big(z^2+c_2y^3\big)=0 \ \ \ \mbox{ on $S_{45}$,}
$$
$$
\nu\big(x^{18}+c_1y^7\big)+\mu y^4\big(z^2+c_2y^3\big)=0  \ \ \ \mbox{ on $S_{81}$,}
$$
$$
\nu\big(x^{26}+c_1y^7\big)+\mu y^4\big(z^2+c_2y^3\big)=0  \ \ \ \mbox{ on $S_{117}$,}
$$
where $c_1$ and $c_2$ are appropriate constants, for $[\nu:\mu]\in\mathbb{P}^1$.
For a general member $C$ in $\mathcal{P}$ we obtain
$$C\cdot D\leqslant\frac{64}{65},$$
which verifies the claim. Notice that the surface $S$ is smooth outside  $C_x$ and  $C_y$.

We now consider  a point $P$ on $C_y$ different from $O_t$. Suppose that the log pair $(S,\frac{65}{64}D)$ is not log canonical at the point $P$.
Recall that $(S,\frac{65e}{64}C_y)$ is log canonical, where $e$ is the positive rational number such that $-K_S\sim_\mathbb{Q}eC_y$, and that the curve $C_y$ is irreducible.
Thus, it follows from \cite[Remark~2.22]{CheltsovShramovUMN} that there exists
an effective $\mathbb{Q}$-divisor $D^\prime$ on the surface $S$ such that
$$
D^\prime\sim_{\mathbb{Q}} -K_{S},
$$
the log pair $(S,\frac{65}{64}D^\prime)$ is not log canonical at the point $P$,
and the support of the divisor $D^\prime$ does not contain the curve $C_y$.
Observe that $$C_y\cdot D^\prime\leqslant \frac{64}{7\cdot65}.$$
This implies that
the log pair $(S,\frac{65}{64}D^\prime)$ is log canonical at the point $P$.  This contradiction shows that the log pair $(S,\frac{65}{64}D)$ is log canonical outside $C_x$.

Let $P$ be a point on $C_x$ other than $O_t$.
We have two cases for the location of $P$, i.e., when $P$ lies on $L_{xz}$ and when it lies on $R_x$. Note that we always have $\frac{65a}{64}<1$ and $\frac{65b}{64}<1$.

We first consider the case where $P$ belongs to $L_{xz}$. Then the log pair $(S,L_{xz}+\frac{65b}{64}R_x+\frac{65}{64}\Delta)$ is log canonical at $P$. Indeed,
$$
\left(bR_x+\Delta\right)\cdot L_{xz}=\big(D-aL_{xz}\big)\cdot L_{xz}=
\left\{%
\aligned
&\frac{6+23a}{190}\leqslant\frac{64}{65\cdot 10}
\ \ \ \text{ for $S_{45}$}\\%
&\frac{8+47a}{37\cdot 18}\leqslant\frac{64}{65\cdot 18}\ \ \  \text{
for  $S_{81}$}\\%
&
\frac{10+71a}{55\cdot 26}\leqslant \frac{64}{65\cdot 26}\ \ \  \text{
for $S_{117}$.}\\%
\endaligned\right.
$$
Lemmas~\ref{lemma:inversion-of-adjunction} or~\ref{lemma:inversion-of-adjunction-orbifold}  imply that $(S,\frac{65}{64}D)$ is log canonical at the point $P$.
If the point $P$ must lie on $R_x$, then we consider
$$
\left(aL_{xz}+\Delta\right)\cdot R_{x}=\big(D-bR_{x}\big)\cdot R_{x}=
\left\{%
\aligned
&\frac{3+8b}{95}\leqslant\frac{64}{65\cdot 5}
\ \ \ \text{ for $S_{45}$}\\%
&\frac{8+20b}{9\cdot 37}\leqslant \frac{64}{65\cdot 9}\ \ \  \text{
for  $S_{81}$}\\%
&
\frac{10+32b}{13\cdot 55}
\leqslant \frac{64}{65\cdot 13}
\ \ \  \text{
for $S_{117}$.}\\%
\endaligned\right.
$$
Lemmas~\ref{lemma:inversion-of-adjunction} or~\ref{lemma:inversion-of-adjunction-orbifold}  then show that $(S,\frac{65}{64}D)$ is log canonical at the point $P$.

Now it is enough to show that $(S,\frac{65}{64}D)$ is log canonical at $O_t$.

Recall that $S_{45}$ (resp. $S_{81}$ snd $S_{117}$) has singularity of type $\frac{1}{19}(2,3)$ (resp. $\frac{1}{37}(2,3)$ and $\frac{1}{55}(2,3)$) at the point $O_t$.
In the chart given by $t=1$,
the surface $S_{45}$ is given by
$$
z^{3}-y^{3}z+x+x^{5}y=0,
$$
the surface $S_{81}$ by
$$
z^{3}-y^{3}z+x+x^{9}y=0,
$$
and the surface $S_{117}$ by
$$
z^{3}-y^{3}z+x+x^{13}y=0.
$$
In a neighborhood  of the point $O_t$, we can consider $y$ and $z$ as local weighted coordinates such that $\mathrm{wt}(y)=2$ and $\mathrm{wt}(z)=3$.

Let $f\colon\widetilde{S}\to S$ be the weighted blow up at the singular point $O_t$ such that $\mathrm{wt}(y)=2$
and~\mbox{$\mathrm{wt}(z)=3$}.
Denote by $E$ the exceptional curve of the blow up $f$.
Then
$$
K_{\widetilde{S}_{45}}\sim_{\mathbb{Q}} f^{*}\big(K_{S_{45}}\big)-\frac{14}{19}E;
$$
$$
K_{\widetilde{S}_{81}}\sim_{\mathbb{Q}} f^{*}\big(K_{S_{81}}\big)-\frac{32}{37}E;
$$
$$
K_{\widetilde{S}_{117}}\sim_{\mathbb{Q}} f^{*}\big(K_{S_{117}}\big)-\frac{10}{11}E.
$$
The surface $S$ has two singular points in $E$.
One  is of type $\frac{1}{2}(1,1)$
and the other  is  of type~\mbox{$\frac{1}{3}(1,1)$}.
Denote the former one by $O_2$ and the latter one by $O_3$.
Observe
\[E^2=
\left\{%
\aligned
&-\frac{19}{6}\ \ \ \text{ on $\widetilde{S}_{45}$},\\%
&-\frac{37}{6}\ \ \ \text{ on $\widetilde{S}_{81}$},\\%
&-\frac{55}{6}\ \ \ \text{ on $\widetilde{S}_{117}$},\\%
\endaligned\right.%
\]
and $E\cong\mathbb{P}^1$.

Let $\widetilde{L}_{xz}$ and  $\widetilde{R}_x$  be the proper transforms of the curve $L_{xz}$  and $R_x$ to the surface $\widetilde{S}$, respectively.
Then
$$
\widetilde{L}_{xz}\sim_{\mathbb{Q}} f^{*}\big(L_{xz}\big)-\frac{3}{c}E, \ \ \ \widetilde{R}_x\sim_{\mathbb{Q}} f^{*}\big(R_x\big)-\frac{6}{c}E,
$$
where $c$ is the index of singularity $O_t$.
The intersection $E\cap\widetilde{L}_{xz}$ consists of the point $O_2$ and the intersection $E\cap\widetilde{R}_x$ consists of a single smooth point.
Note that $\widetilde{L}_{xz}\cdot E=\frac{1}{2}$ and the curves $E$ and $\widetilde{R}_x$ intersect transversally.

Recall that $D=aL_{xz}+bR_x+\Delta$.
Denote by $\widetilde{\Delta}$ be the proper transform of the $\mathbb{Q}$-divisor $\Delta$ on the surface $\widetilde{S}$.
Then
$$
\widetilde{\Delta}\sim_{\mathbb{Q}} f^{*}\big(\Delta\big)-mE
$$
for some non-negative rational number $m$. To estimate $m$, consider the intersection
$$
0\leqslant\widetilde{R}_{x}\cdot\widetilde{\Delta}=\widetilde{R}_{x}\cdot\Big(f^{*}\big(\Delta\big)-mE\Big)=R_{x}\cdot\Delta-m.$$
Applying  \eqref{equation:3-intersection}, we are able to obtain
\begin{equation}\label{equation:m}
m\leqslant\left\{\aligned
&\frac{6}{5\cdot 19}-\frac{3a}{19}+\frac{8b}{5\cdot 19}\leqslant \frac{6}{5\cdot 19}+\frac{8b}{5\cdot 19}\leqslant\frac{26}{285}
\ \ \ \text{ for $S_{45}$,}\\%
&\frac{8}{9\cdot 37}-\frac{3a}{37}+\frac{20b}{9\cdot 37}\leqslant \frac{8}{9\cdot 37}+\frac{20b}{9\cdot 37}\leqslant \frac{4}{111}\ \ \  \text{
for  $S_{81,}$}\\%
&
\frac{2}{11\cdot 13}-\frac{3a}{55}+\frac{32b}{13\cdot 55}\leqslant \frac{2}{11\cdot 13}+\frac{32b}{13\cdot 55}\leqslant\frac{634}{17875}
\ \ \  \text{
for $S_{117}$.}\\%
\endaligned\right.
\end{equation}

We now suppose that the log pair $(S,\frac{65}{64}D)$ is not log canonical at~$O_t$.
Put $\lambda=\frac{65}{64}$. Then
$$
K_{\widetilde{S}}+\lambda a\widetilde{L}_{xz}+\lambda b\widetilde{R}_{x}+\lambda\widetilde{\Delta}+\mu E\sim_{\mathbb{Q}}f^*\big(K_S+\lambda D\big),
$$
where $$\mu =
\left\{\aligned
&\frac{3\lambda a}{19}+\frac{6\lambda b}{19}+\lambda m+\frac{14}{19}\ \ \ \text{ for $S_{45}$,}\\%
&\frac{3\lambda a}{37}+\frac{6\lambda b}{37}+\lambda m+\frac{32}{37}\ \ \  \text{
for  $S_{81,}$}\\%
&
\frac{3\lambda a}{55}+\frac{6\lambda b}{55}+
\lambda m+\frac{10}{11}
\ \ \  \text{
for $S_{117}$.}\\%
\endaligned\right.
$$
Thus, the log pair
\begin{equation}
\label{equation:3-log-pull-back}
\left(\widetilde{S},\lambda a\widetilde{L}_{xz}+\lambda b\widetilde{R}_{x}+\lambda\widetilde{\Delta}+\mu E\right)
\end{equation}
is not log canonical at some point $O\in E$.
Using \eqref{equation:m} and bounds for $b$, we can easily check
$$
\mu\leqslant
\left\{\aligned
&\frac{3\lambda a}{19}+\frac{6\lambda b}{19}+\frac{6\lambda}{95}-\frac{3a\lambda}{19}+\frac{8\lambda b}{95}+\frac{14}{19}=\frac{2\lambda b}{5}+\frac{6\lambda}{95}+\frac{14}{19}\leqslant 1\ \ \ \text{ for $S_{45}$,}\\%
&\frac{3\lambda a}{37}+\frac{6\lambda b}{37}+\frac{8\lambda}{9\cdot 37}-\frac{3\lambda a}{37}+\frac{20\lambda b}{9\cdot 37}+\frac{32}{37} =\frac{2\cdot 29\lambda b}{3\cdot 37}+\frac{8\lambda}{9\cdot 37}+\frac{32}{37}\leqslant 1\ \ \  \text{
for  $S_{81,}$}\\%
&
\frac{3\lambda a}{55}+\frac{6\lambda b}{55}+\frac{2\lambda}{11\cdot 13}
-\frac{3\lambda a}{55}+\frac{32\lambda b}{13\cdot 55}+\frac{10}{11}=
\frac{2\lambda b}{13}+\frac{2\lambda}{11\cdot 13}+\frac{10}{11}\leqslant 1
\ \ \  \text{
for $S_{117}$.}\\%
\endaligned\right.
$$
If $O=E\cap \widetilde{R}_x$, then we apply Lemma~\ref{lemma:inversion-of-adjunction} to \eqref{equation:3-log-pull-back} and $E$. This yields
$$
\lambda b+\lambda\widetilde{\Delta}\cdot E=\Big(\lambda b\widetilde{R}_x+\lambda\widetilde{\Delta}\Big)\cdot E>1,
$$
so that we  could obtain absurd inequalities
$$
\frac{64}{65}=\frac{1}{\lambda}<b+\widetilde{\Delta}\cdot E= b+\frac{cm}{6}\leqslant
\left\{\aligned
&\frac{1}{3}+\frac{19}{60}=\frac{13}{20}\ \ \ \text{ for $S_{45}$,}\\%
&\frac{1}{5}+\frac{37}{6\cdot 25}=\frac{67}{150}\ \ \  \text{
for  $S_{81,}$}\\%
&
\frac{12}{25}+\frac{1}{3}= \frac{61}{75}
\ \ \  \text{
for $S_{117}$,}\\%
\endaligned\right.
$$
where $c$ is the index of the singularity $O_t$.
The inequality
$$
\widetilde{\Delta}\cdot E= \frac{cm}{6}\leqslant
\left\{\aligned
&\frac{13}{45} \ \ \ \text{ for $S_{45}$,}\\%
&\frac{2}{9}\ \ \  \text{
for  $S_{81,}$}\\%
&
\frac{317}{975}
\ \ \  \text{
for $S_{117}$.}\\%
\endaligned\right\}\leqslant \frac{1}{3\lambda} =\frac{64}{3\cdot 65}
$$
implies  that $O=O_2$.
However, using \eqref{equation:m} and Lemma~\ref{lemma:inversion-of-adjunction-orbifold} (applied to \eqref{equation:3-log-pull-back} and $E$),
we conclude that the log pair \eqref{equation:3-log-pull-back} is log canonical everywhere since
$$
\left(a\widetilde{L}_{xz}+\widetilde{\Delta}\right)\cdot E=\frac{a}{2}+\widetilde{\Delta}\cdot E= \frac{a}{2}+\frac{cm}{6}\leqslant
\left\{\aligned
&\frac{1}{5}+\frac{4b}{15}\leqslant\frac{13}{45}\ \ \ \text{ for $S_{45}$,}\\%
&\frac{4}{27}+\frac{10b}{27}\leqslant\frac{2}{9}\ \ \  \text{ for  $S_{81,}$}\\%
& \frac{5}{39}+\frac{16b}{39}\leqslant\frac{317}{975} \ \ \  \text{ for $S_{117}$,}\\%
\endaligned\right\} \leqslant \frac{1}{2\lambda}=\frac{32}{65}
$$
This completes the proof. \end{proof}

\begin{corollary}
\label{corollary:3}
The $\delta$-invariant of $S$ is at least $\frac{65}{64}$.
\end{corollary}
\begin{proof}
This immediately follows from Proposition~\ref{proposition:3} and Lemma~\ref{lemma:a-b-bound}.
\end{proof}

\end{document}